\documentclass[reqno,12pt]{article}
\usepackage{a4wide,color,eucal,enumerate,mathrsfs}
\usepackage[normalem]{ulem}
\usepackage{amsmath,amssymb,epsfig,bbm}
\numberwithin{equation}{section}

\usepackage[latin1]{inputenc}

\newcommand{\C}{\mathbb{C}}

\newcommand{\N}{\mathbb{N}}

\newcommand{\R}{\mathbb{R}}

\newcommand{\mm}{{\mbox{\boldmath$m$}}}

\newcommand{\ggamma}{{\mbox{\boldmath$\gamma$}}}

\newcommand{\ppi}{{\mbox{\boldmath$\pi$}}}

\newcommand{\sfd}{{\sf d}}

\newcommand{\rme}{{\mathrm e}}

\newcommand{\Kliminf}{K\kern-3pt-\kern-2pt\mathop{\rm lim\,inf}\limits}  
\newcommand{\Lip}{\mathop{\rm Lip}\nolimits}          
\renewcommand{\d}{{\mathrm d}}
\newcommand{\dt}{{\d t}}

\newcommand{\restr}[1]{\lower3pt\hbox{$|_{#1}$}}

\newcommand{\Leb}[1]{{\mathscr L}^{#1}}      

\newcommand{\eps}{\varepsilon}  
\newcommand{\nchi}{{\raise.3ex\hbox{$\chi$}}}

\newcommand{\Pc}[2]{\overline{#1}\kern-2pt^{\vphantom 0}_{#2}}

\newcommand{\Probabilities}[1]{\mathscr P(#1)}          

\newenvironment{proof}{\removelastskip\par\medskip   
\noindent{\em Proof.}
\rm}{\penalty-20\null\hfill$\square$\par\medbreak}

\newtheorem{theorem}{Theorem}[section]

\newtheorem{lemma}[theorem]{Lemma}
\newtheorem{proposition}[theorem]{Proposition}
\newtheorem{definition}[theorem]{Definition}

\newtheorem{remark}[theorem]{Remark}

\newcommand{\prob}{\Probabilities}

\newcommand{\lims}{\varlimsup}

\newcommand{\e}{{\rm{e}}}                           
\newcommand{\fr}{\hfill$\blacksquare$}                      
\newcommand{\sppi}{{\mbox{\scriptsize\boldmath$\pi$}}}      

\newcommand{\relgradq}[2]{|\nabla #1|_{*,#2}} 
\newcommand{\weakgrad}[1]{|\nabla #1|_w} 
\newcommand{\weakgradq}[2]{|\nabla #1|_{w,#2}} 

\renewcommand{\mm}{\mathfrak m}

\renewcommand{\C}{{\rm Ch}}

\title{Density of Lipschitz functions and equivalence of weak gradients in metric measure spaces}
\begin{document}
\author{Luigi Ambrosio\
   \thanks{Scuola Normale Superiore, Pisa. email: \textsf{l.ambrosio@sns.it}}
   \and
   Nicola Gigli\
   \thanks{Nice University. email: \textsf{nicola.gigli@unice.fr}}
 \and
   Giuseppe Savar\'e\
   \thanks{Universit\`a di Pavia. email: \textsf{giuseppe.savare@unipv.it}}
   }

\maketitle

\begin{abstract}
We compare several notion of weak (modulus of) gradient in metric measure spaces and prove their equivalence. Using tools from optimal transportation theory we prove density in energy of Lipschitz maps independently of doubling and Poincar\'e assumptions on the metric measure space.
\end{abstract}

\tableofcontents

%
%

\section{Introduction}

In the last few years a great attention has been devoted to the
theory of Sobolev spaces $W^{1,q}$ on metric measure spaces
$(X,\sfd,\mm)$, see for instance \cite{Heinonen07} and
\cite{Hajlasz-Koskela}
for an overview on this subject. These definitions of Sobolev spaces
usually come with a weak definition of modulus of gradient, in
particular the notion of $q$-upper gradient has been introduced in
\cite{Koskela-MacManus} and used in \cite{Shanmugalingam00} for a
Sobolev space theory. Also, in \cite{Shanmugalingam00} the notion of
minimal $q$-upper gradient has been proved to be equivalent to the
notion of relaxed upper gradient arising in Cheeger's paper
\cite{Cheeger00}.

In this paper we consider a notion of gradient $|\nabla f|_{*,q}$
stronger than the one of \cite{Cheeger00}, because in the
approximation procedure we use Lipschitz functions and their slopes
as upper gradients, and a notion of $q$-weak upper gradient $|\nabla
f|_{w,q}$ weaker than the one of \cite{Shanmugalingam00}, and prove
their equivalence. As a consequence all four notions of gradient
turn out to be equivalent. A byproduct of our equivalence result is
the following density in energy of Lipschitz functions: if $f\in
L^q(X,\mm)$ has a $q$-weak upper gradient $|\nabla f|_{w,q}$ in
$L^q(X,\mm)$, then there exist Lipschitz functions $f_n$ convergent
to $f$ in $L^q(X,\mm)$ satisfying (here $|\nabla f_n|$ is the slope
of $f_n$)
\begin{equation}\label{densitylip}
\lim_{n\to\infty}\int_X\bigl||\nabla f_n|-|\nabla
f|_{w,q}\bigr|^q\,\d\mm=0.
\end{equation}
Notice that we can use Mazur's lemma to improve this convergence to
strong convergence in $W^{1,q}(X,\sfd,\mm)$, as soon as this space
is reflexive; this happens for instance in the context of the spaces
with Riemannian Ricci bounds from below considered in
\cite{Ambrosio-Gigli-Savare11bis}, with $q=2$.

We emphasize that our density result does not depend on doubling and
Poincar\'e assumptions on the metric measure structure; as it is
well known (see Theorem~4.14 and Theorem~4.24 in \cite{Cheeger00}),
these assumptions ensure the density in Sobolev norm of Lipschitz
functions, even in the Lusin sense (i.e. the Lipschitz approximating
functions $f_n$ coincide with $f$ on larger and larger sets). On the
other hand, the density in energy \eqref{densitylip} suffices for
many purposes, for instance the extension by approximation, from
Lipschitz to Sobolev functions, of functional inequalities like the
Poincar\'e or Sobolev inequality. For instance, our result can be used to 
show that if $(X,\sfd)$ is complete and separable and $\mm$ is a Borel measure finite on 
bounded sets, then the Poincar\'e inequality 
$$
\int_{B_r(x)}|f(y)-f_{B_r(x)}|\,\d\mm(y)\leq Cr\int_{B_{\lambda r}(x)}|\nabla f|(y)\,d\mm(y)
$$
holds for all $f:X\to\R$ Lipschitz on bounded sets if and only if it holds in the form 
$$
\int_{B_r(x)}|f(y)-f_{B_r(x)}|\,\d\mm(y)\leq Cr\int_{B_{\lambda r}(x)}g(y)\,d\mm(y)
$$
for all pairs $(f,g)$ with $f$ Borel and $g$
upper gradient of $f$.
This equivalence was proven in \cite{Heinonen-Koskela99} for proper, 
quasiconvex and doubling metric measure spaces, while in 
\cite{Koskela_removable} (choosing $X=\R^n\setminus E$ for suitable
 compact sets $E$) it is proven that completeness of the space can't 
be dropped.

The new notions of gradient, as well as their equivalence, have been
proved in \cite{Ambrosio-Gigli-Savare11} in the case $q=2$, see
Corollary~6.3 therein. Here we extend the result to general
exponents $q\in (1,\infty)$ and we give a presentation more focussed
on the equivalence problem. While the traditional proof of density
of Lipschitz functions relies on Poincar\'e inequality, maximal
functions and covering arguments to construct the ``optimal''
approximating Lipschitz functions $f_n$, our proof is more indirect
and provides the approximating functions using the $L^2$-gradient
flow of $\C_q(f):=q^{-1}\int_X|\nabla f|_{*,q}^q\,\d\mm$ and the
analysis of the dissipation rate along this flow of a suitable
``entropy'' $\int\Phi_q(f)\,\d\mm$ (in the case $q=2$, $\Phi(z)=z\log z$). This way we prove that $|\nabla
f|_{w,q}=|\nabla f|_{*,q}$ $\mm$-a.e., and then \eqref{densitylip}
follows by a general property of the minimal $q$-relaxed slope
$|\nabla f|_{*,q}$, see Proposition~\ref{prop:easy}.

The paper is organized as follows. In Section~2 we recall some
preliminary facts on absolutely continuous curves and gradient
flows. We also introduce the $p$-th Wasserstein distance and the
so-called superposition principle, that allows to pass from an
``Eulerian'' formulation (i.e. in terms of a curve of measures or a
curve of probability densities) to a ``Lagrangian'' one, namely a
probability measure in the space of absolutely continuous paths;
this will be the only tool from optimal transportation theory used
in the paper.\\* In Section~3 we study the pointwise properties of
the Hopf-Lax semigroup
$$
Q_tf(x):=\inf_{y\in X}f(y)+\frac{\sfd^p(x,y)}{pt^{p-1}}.
$$
In comparison with Section~3 of \cite{Ambrosio-Gigli-Savare11},
dealing with the case $p=2$, we consider for the sake of simplicity
only locally compact spaces and finite distances, but the
proofs can be modified to deal with more general cases, see also
Section~8. The results of this section overlap with those of the
forthcoming paper \cite{Gozlan} by Gozlan, Roberto and Samson, where
the HL semigroup is used in connection with the proof of transport
entropy inequalities.\\* In Section 4 we introduce the four
definitions of gradients we will be dealing with, namely:
\begin{itemize}
\item[(1)]
 the Cheeger gradient $|\nabla f|_{C,q}$ of \cite{Cheeger00} arising from the relaxation
of upper gradients;
\item[(2)]
 the minimal relaxed slope $|\nabla f|_{*,q}$ of
\cite{Ambrosio-Gigli-Savare11} arising from the relaxation of the
slope of Lipschitz functions;
\item[(3)]
 the minimal $q$-upper gradient
$|\nabla f|_{S,q}$ of \cite{Koskela-MacManus,Shanmugalingam00},
based on the validity of the upper gradient property out of a ${\rm
Mod}_q$-null set of curves;
\item[(4)] the minimal $q$-weak upper gradient of \cite{Ambrosio-Gigli-Savare11},
based on the validity of the upper gradient property out of a
$q$-null set of curves.
\end{itemize}
While presenting these definitions we will point out natural
relations between them, that lead to the chain of inequalities
$$
|\nabla f|_{w,q}\leq|\nabla f|_{S,q}\leq|\nabla f|_{C,q}\leq |\nabla
f|_{*,q}\qquad\text{$\mm$-a.e. in $X$,}
$$
with the concepts of \cite{Ambrosio-Gigli-Savare11} at the extreme
sides.

Section~5 contains some basically well known properties of weak
gradients, namely chain rules and stability under weak convergence.
Section~6 contains the basic facts we shall need on the gradient flow of the lower
semicontinuous functional $\C_q$ we need, in particular the entropy
dissipation rate
$$
\frac{\d}{\d t}\int_X\Phi(f_t)\,\d\mm=-\int_X\Phi''(f_t)|\nabla
f_t|^q_{*,q}\,\d\mm
$$
along this gradient flow.

In Section~7 we prove the equivalence of gradients. Starting from a
function $f$ with $|\nabla f|_{w,q}\in L^q(X,\mm)$ we approximate it
by the gradient flow of $f_t$ of $\C_q$ starting from $f$ and we use
the weak upper gradient property to get
$$
\limsup_{t\downarrow 0}\frac1t\int_0^t\int_X\frac{|\nabla
f_s|_{*,q}^q}{f_s^{p-1}}\,d\mm\d s\leq\int_X\frac{|\nabla
f|_{w,q}^q}{f^{p-1}}\,d\mm
$$
where $p=q/(q-1)$ is the dual exponent of $q$. Using the stability
properties of the relaxed gradients we eventually get $|\nabla
f|_{*,q}\leq|\nabla f|_{*,w}$ $\mm$-a.e. in $X$.

Finally, Section~8 discusses some potential extensions of the
results of this paper: we indicate how spaces which are not locally
compact and measures that are locally finite can be achieved. Other
extensions require probably a separate investigation, as the case of
Orlicz spaces and the limiting case $q=1$, corresponding to
$W^{1,1}$ and $BV$ spaces. In this latter case the lack of
reflexivity of $L^1(X,\mm)$ poses problems even in the definition of
the minimal gradients and we discuss this very briefly.

\smallskip
\noindent {\bf Acknowledgement.} The authors acknowledge the support
of the ERC ADG GeMeThNES. The authors also thank P.Koskela for
useful comments during the preparation of the paper.

\section{Preliminary notions}\label{sec:preliminary}

In this section we introduce some notation and recall a few basic
facts on absolutely continuous functions, gradient flows of convex
functionals and optimal transportation, see also
\cite{Ambrosio-Gigli-Savare08}, \cite{Villani09} as general
references.

\subsection{Absolutely continuous curves and slopes}
Let $(X,\sfd)$ be a metric space, $J\subset\R$ a closed interval and
$J\ni t\mapsto x_t\in X$. We say that $(x_t)$ is \emph{absolutely
continuous} if
$$
\sfd(x_s,x_t)\leq\int_s^tg(r)\,\d r\qquad\forall s,\,t\in J,\,\,s<t
$$
for some $g\in L^1(J)$. It turns out that, if $(x_t)$ is absolutely
continuous, there is a minimal function $g$ with this property,
called \emph{metric speed}, denoted by $|\dot{x}_t|$ and given for
a.e. $t\in J$ by
$$
|\dot{x}_t|=\lim_{s\to t}\frac{\sfd(x_s,x_t)}{|s-t|}.
$$
See \cite[Theorem~1.1.2]{Ambrosio-Gigli-Savare08} for the simple
proof.

We will denote by $C([0,1],X)$ the space of continuous curves from
$[0,1]$ to $(X,\sfd)$ endowed with the $\sup$ norm. The set
$AC^p([0,1],X)\subset C([0,1],X)$ consists of all absolutely
continuous curves $\gamma$ such that $\int_0^1|\dot\gamma_t|^p\,\d
t<\infty$: it is the countable union of the sets $\{\gamma:\
\int_0^1|\dot\gamma_t|^p\,\d t\leq n\}$, which are easily seen to be
closed if $p>1$. Thus $AC^p([0,1],X)$ is a Borel subset of
$C([0,1],X)$. The \emph{evaluation maps} $\e_t:C([0,1],X)\to X$ are
defined by
\[
\e_t(\gamma):=\gamma_t,
\]
and are clearly continuous.

Given $f:X\to\R$, we define \emph{slope} (also called local
Lipschitz constant) by
$$
|\nabla f|(x):=\lims_{y\to x}\frac{|f(y)-f(x)|}{\sfd(y,x)}.
$$
For $f,\,g:X\to\R$ Lipschitz it clearly holds
\begin{subequations}
\begin{align}
\label{eq:subadd}
|\nabla(\alpha f+\beta g)|&\leq|\alpha||\nabla f|+|\beta||\nabla g|\qquad\forall \alpha,\beta\in\R,\\
\label{eq:leibn} |\nabla (fg)|&\leq |f||\nabla g|+|g||\nabla f|.
\end{align}
\end{subequations}

{We shall also need the following calculus lemma.}
\begin{lemma}\label{lem:Fibonacci}
{
Let $f:(0,1)\to\R$, $q\in [1,\infty]$, $g\in L^q(0,1)$ nonnegative be satisfying
$$
|f(s)-f(t)|\leq\bigl|\int_s^t g(r)\,\d r\bigr|\qquad\text{for $\Leb{2}$-a.e. $(s,t)\in (0,1)^2$.}
$$
Then $f\in W^{1,q}(0,1)$ and $|f'|\leq g$ a.e. in $(0,1)$.}
\end{lemma}

\begin{proof} 
Let $N\subset (0,1)^2$ be the $\Leb{2}$-negligible subset where the above 
inequality fails. Choosing $s\in (0,1)$, whose existence is ensured by Fubini's theorem, such that
$(s,t)\notin N$ for a.e. $t\in (0,1)$, we obtain that $f\in L^\infty(0,1)$. 
Since the set $\{(t,h)\in (0,1)^2:\ (t,t+h)\in N\cap (0,1)^2\}$
is $\Leb{2}$-negligible as well, we can apply Fubini's theorem to obtain that for a.e. $h$ it holds
$(t,t+h)\notin N$ for a.e. $t\in (0,1)$. Let $h_i\downarrow 0$ with this property and use
the identities
$$
\int_0^1f(t)\frac{\phi(t+h)-\phi(t)}{h}\,\d t=-\int_0^1\frac{f(t-h)-f(t)}{-h}\phi(t)\,\d t
$$
with $\phi\in C^1_c(0,1)$ and $h=h_i$ sufficiently small to get
$$
\biggl|\int_0^1f(t)\phi'(t)\,\d t\biggr|\leq\int_0^1g(t)|\phi(t)|\,\d t.
$$
It follows that the distributional derivative of $f$ is a signed
measure $\eta$ with finite total variation which satisfies
\begin{displaymath}
  -\int_0^1f\phi'\,\d t=\int_0^1 \phi\,\d\eta,\qquad
  \Bigl|\int_0^1 \phi\,\d\eta\Bigr|\le \int_0^1g|\phi|\,\d
t\quad\text{for every }\phi\in C^1_c(0,1);
\end{displaymath}
therefore $\eta$ is absolutely continuous with respect to the Lebesgue
measure with $|\eta|\le g\Leb 1$. 
This gives the $W^{1,1}(0,1)$ regularity and, at the same time, the inequality
$|f'|\leq g$ a.e. in $(0,1)$. The case $q>1$ immediately follows
by applying this inequality when $g\in L^q(0,1)$.
\end{proof}

\subsection{Gradient flows of convex functionals}

Let $H$ be an Hilbert space, $\Psi:H\to\R\cup\{+\infty\}$ convex and
lower semicontinuous and $D(\Psi)=\{\Psi<\infty\}$ its finiteness
domain. Recall that a gradient flow $x:(0,\infty)\to H$ of $\Psi$ is
a locally absolutely continuous map with values in $D(\Psi)$
satisfying
$$
-\frac{\d}{\dt}x_t\in\partial^-\Psi(x_t)\qquad\text{for a.e. $t\in
(0,\infty)$.}
$$
Here $\partial^-\Psi(x)$ is the subdifferential of $\Psi$, defined
at any $x\in D(\Psi)$ by
$$
\partial^-\Psi(x):=\left\{p\in H^*:\ \Psi(y)\geq\Psi(x)+\langle
p,y-x\rangle\,\,\forall y\in H\right\}.
$$

We shall use the fact that for all $x_0\in\overline{D(\Psi)}$ there
exists a unique gradient flow $x_t$ of $\Psi$ starting from $x_0$,
i.e. $x_t\to x_0$ as $t\downarrow 0$, and that $t\mapsto\Psi(x_t)$
is nonincreasing and locally absolutely continuous in $(0,\infty)$.
In addition, this unique solution exhibits a regularizing effect,
namely $-\tfrac{\d}{\d t}x_t$ is for a.e. $t\in (0,\infty)$ the
element of minimal norm in
$\partial^-\Psi(x_t)$. 

\subsection{The space $(\prob X,W_p)$ and the superposition principle}

Let $(X,\sfd)$ be a complete and separable metric space and $p\in
[1,\infty)$. We use the notation $\prob X$ for the set of all Borel
probability measures on $X$. Given $\mu,\,\nu\in\prob X$, we define
the Wasserstein (extended) distance $W_p(\mu,\nu)\in [0,\infty]$
between them as
\[
W_p^p(\mu,\nu):=\min\int \sfd^p(x,y)\,\d\ggamma(x,y).
\]
Here the minimization is made in the class $\Gamma(\mu,\nu)$ of all
probability measures $\ggamma$ on $X\times X$ such that
$\pi^1_\#\ggamma=\mu$ and $\pi^2_\#\ggamma=\nu$, where
$\pi^i:X\times X\to X$, $i=1,\,2$, are the coordinate projections
and $f_\#:\prob{Y}\to\prob{Z}$ is the push-forward operator induced
by a Borel map $f:Y\to Z$.

An equivalent definition of $W_p$ comes from the dual formulation of
the transport problem. In the case when $(X,\sfd)$ has finite
diameter the dual formulation takes the simplified form
\begin{equation}
\label{eq:dualitabase} \frac1pW_p^p(\mu,\nu)=\sup_{\psi\in{\rm
Lip}(X)}\int\psi\, \d\mu+\int \psi^c\,\d\nu,
\end{equation}
where the $c$-transform $\psi^c$ is defined by
\[
\psi^c(y):=\inf_{x\in X}\frac{\sfd^p(x,y)}p-\psi(x).
\]

We will need the following result, proved in \cite{Lisini07}: it
shows how to associate to an absolutely continuous curve $\mu_t$
w.r.t. $W_p$ a plan $\ppi\in\prob{C([0,1],X)}$ representing the
curve itself (see also \cite[Theorem~8.2.1]{Ambrosio-Gigli-Savare08}
for the Euclidean case).

\begin{proposition}[Superposition principle]\label{prop:lisini}
Let $(X,\sfd)$ be a complete and separable metric space with $\sfd$
bounded, $p\in (1,\infty)$ and let $\mu_t\in AC^p\bigl([0,T];(\prob X,W_p)\bigr)$.
Then there exists $\ppi\in\prob{C([0,1],X)}$,
concentrated on $AC^p([0,1],X)$, such that $(\e_t)_\sharp\ppi=\mu_t$
for any $t\in[0,T]$ and
\begin{equation}\label{eq:Lisini}
\int|\dot\gamma_t|^p\,\d\ppi(\gamma)=|\dot\mu_t|^p\qquad \text{for
a.e. $t\in [0,T]$.}
\end{equation}
\end{proposition}

\section{Hopf-Lax formula and Hamilton-Jacobi equation}\label{sec:hopflax}

Aim of this section is to study the properties of the Hopf-Lax
formula in a metric setting and its relations with the
Hamilton-Jacobi equation. Here we assume for simplicity that
$(X,\sfd)$ is a compact metric space, see Section~\ref{sextensions}
for a more general discussion. Notice that there is no reference
measure $\mm$ here. We fix a power $p\in (1,\infty)$ and denote by
$q$ the dual exponent.

Let $f:X\to\R$ be a Lipschitz function. For $t>0$ define
\begin{equation}\label{eq:Nicola1}
F(t,x,y):=f(y)+\frac{\sfd^p(x,y)}{pt^{p-1}},
\end{equation}
and the function $Q_tf:X\to\R$ by
\begin{equation}\label{eq:Nicola2}
Q_tf(x):=\inf_{y\in X}F(t,x,y)=\min_{y\in X}F(t,x,y).
\end{equation}
Also, we introduce the functions $D^+,\,D^-:X\times(0,\infty)\to\R$
as
\begin{equation}\label{eq:defdpm}
\begin{split}
D^+(x,t)&:=\max\, \sfd(x,y),\\
D^-(x,t)&:=\min\, \sfd(x,y),\\
\end{split}
\end{equation}
where, in both cases, the $y$'s vary among all minima of
$F(t,x,\cdot)$. We also set $Q_0f=f$ and $D^\pm(x,0)=0$. Arguing as
in \cite[Lemma~3.1.2]{Ambrosio-Gigli-Savare08} it is easy to check
that the map $[0,\infty)\ni(t,x)\mapsto Q_tf(x)$ is continuous.
Furthermore, the fact that $f$ is Lipschitz easily yields
\begin{equation}
\label{eq:boundD} D^-(x,t)\leq D^+(x,t)\leq t(p\Lip(f))^{1/(p-1)}.
\end{equation}

\begin{proposition}[Monotonicity of $D^\pm$]\label{prop:dmon}
For all $x\in X$ it holds
\begin{equation}\label{eq:basic_mono}
D^+(x,t)\leq D^-(x,s)\qquad 0\leq t< s.
\end{equation}
As a consequence, $D^+(x,\cdot)$ and $D^-(x,\cdot)$ are both
nondecreasing, and they coincide with at most countably many
exceptions in $[0,\infty)$.
\end{proposition}
\begin{proof}
Fix $x\in X$. For $t=0$ there is nothing to prove. Now pick $0<t<s$
and choose $x_t$ and $x_s$ minimizers of $F(t,x,\cdot)$ and
$F(s,x,\cdot)$ respectively, such that $\sfd(x,x_t)=D^+(x,t)$ and $
\sfd(x,x_s)=D^-(x,s)$. The minimality of $x_t,\,x_s$ gives
\[
\begin{split}
f(x_t)+\frac{\sfd^p(x_t,x)}{pt^{p-1}}&\leq f(x_s)+\frac{\sfd^p(x_s,x)}{pt^{p-1}}\\
f(x_s)+\frac{\sfd^p(x_s,x)}{ps^{p-1}}&\leq
f(x_t)+\frac{\sfd^p(x_t,x)}{ps^{p-1}}.
\end{split}
\]
Adding up and using the fact that $\tfrac1t\geq\tfrac 1s$ we deduce
\[
D^+(x,t)=\sfd(x_t,x)\leq \sfd(x_s,x)= D^-(x,s),
\]
which is \eqref{eq:basic_mono}. Combining this with the inequality
$D^-\leq D^+$ we immediately obtain that both functions are
nonincreasing. At a point of right continuity of $D^-(x,\cdot)$ we
get
$$
D^+(x,t)\leq\inf_{s>t}D^-(x,s)=D^-(x,t).
$$
This implies that the two functions coincide out of a countable set.
\end{proof}

Next, we examine the semicontinuity properties of $D^\pm$. These
properties imply that points $(x,t)$ where the equality
$D^+(x,t)=D^-(x,t)$ occurs are continuity points for both $D^+$ and
$D^-$.

\begin{proposition}[Semicontinuity of $D^\pm$]
$D^+$ is upper semicontinuous and $D^-$ is lower semicontinuous in
$X\times [0,\infty)$.
\end{proposition}
\begin{proof}
We prove lower semicontinuity of $D^-$, the proof of upper
semicontinuity of $D^+$ being similar. Let $(x_i,t_i)$ be any
sequence converging to $(x,t)$ such that the limit of $D^-(x_i,t_i)$
exists and assume that $t>0$ (the case $t=0$ is trivial). For every
$i$, let $(y_i)$ be a minimum of $F(t_i,x_i,\cdot)$ for which
$\sfd(y_i,x_i)=D^-(x_i,t_i)$, so that
\[
f(y_i)+\frac{\sfd^p(y_i,x_i)}{pt_i^{p-1}}=Q_{t_i}f(x_i).
\]
The continuity of $(x,t)\mapsto Q_tf(x)$ gives that
$\lim_iQ_{t_i}f(x_i)=Q_tf(x)$, thus
\[
\lim_{i\to\infty} f(y_i)+\frac{\sfd^p(y_i,x)}{pt^{p-1}}=Q_tf(x),
\]
that is: $i\mapsto y_i$ is a minimizing sequence for $F(t,x,\cdot)$.
Since $(X,\sfd)$ is compact, possibly passing to a subsequence, not
relabeled, we may assume that $(y_i)$ converges to some $y$ as
$i\to\infty$. Therefore
\[
D^-(x,t)\leq
\sfd(x,y)=\lim_{i\to\infty}\sfd(x_i,y_i)=\lim_{i\to\infty}D^-(x_i,t_i).
\]
\end{proof}

\begin{proposition}[Time derivative of
$Q_tf$]\label{prop:timederivative} The map $t\mapsto Q_tf$ is
Lipschitz from $[0,\infty)$ to $C(X)$ and, for all $x\in X$, it
satisfies:
\begin{equation}\label{eq:Dini1}
\frac{\d}{\d
t}Q_tf(x)=-\frac{1}{q}\bigl[\frac{D^{\pm}(x,t)}{t}\bigr]^p,
\end{equation}
for any $t>0$, with at most countably many exceptions.
\end{proposition}
\begin{proof}
Let $t<s$ and $x_t$, $x_s$ be minima of $F(t,x,\cdot)$
and $F(s,x,\cdot)$. We have
\[
\begin{split}
Q_sf(x)-Q_tf(x)&\leq F(s,x,x_t)-F(t,x,x_t)=\frac{\sfd^p(x,x_t)}{p}\frac{t^{p-1}-s^{p-1}}{t^{p-1}s^{p-1}},\\
Q_sf(x)-Q_tf(x)&\geq
F(s,x,x_s)-F(t,x,x_s)=\frac{\sfd^p(x,x_s)}{p}\frac{t^{p-1}-s^{p-1}}{t^{p-1}s^{p-1}},
\end{split}
\]
which gives that $t\mapsto Q_tf(x)$ is Lipschitz in $(\eps,\infty)$
for any $\eps>0$ uniformly with respect to $x\in X$. Also, dividing
by $(s-t)$ and taking Proposition~\ref{prop:dmon} into account, we
get \eqref{eq:Dini1}. Now notice that from \eqref{eq:boundD} we get
that $q|\frac{\d}{\d t}Q_tf(x)|\leq p^q[\Lip(f)]^q$ for any $x\in X$
and a.e. $t$, which, together with the pointwise convergence of
$Q_tf$ to $f$ as $t\downarrow 0$, yields that $t\mapsto Q_tf\in
C(X)$ is Lipschitz in $[0,\infty)$.
\end{proof}

In the next proposition we bound from above the slope of $Q_tf$ at $x$ with
$|D^+(x,t)/t|^{p-1}$; actually we shall prove a more precise statement, in connection with \S\ref{sec:improveslope},
which involves the \emph{asymptotic Lipschitz constant}
\begin{equation}\label{eq:asymlip}
{\rm Lip}_a(f,x):=\inf_{r>0}{\rm Lip}\bigl(f,B_r(x)\bigr)=\lim_{r\downarrow 0}{\rm Lip}\bigl(f,B_r(x)\bigr).
\end{equation}
Notice that ${\rm Lip}(f)\geq {\rm Lip}_a(f,x)\geq |\nabla f|^*(x)$, where $|\nabla f|^*$ is the upper semicontinuous envelope 
of the slope of $f$. The second inequality is easily seen to be an equality in length spaces.

\begin{proposition}[Bound on the asymptotic Lipschitz constant of $Q_tf$]\label{prop:slopesqt}
For $(x,t)\in X\times (0,\infty)$ it holds:
\begin{equation}
\label{eq:hjbss} {\rm Lip}_a(Q_tf,x)\leq
\bigl[\frac{D^+(x,t)}t\bigr]^{p-1}.
\end{equation}
\end{proposition}
In particular ${\rm Lip}(Q_t(f))\leq p{\rm Lip}(f)$.
\begin{proof} Fix $y,\,z\in X$, $t\in (0,\infty)$ and a minimizer $\bar y$ for 
$F(t,y,\cdot)$. Since it holds
\[
\begin{split}
Q_tf(z)-Q_tf(y)&\leq F(t,z,\bar y)-F(t,y,\bar y)=
f(\bar y)+\frac{\sfd^p(z,\bar y)}{pt^{p-1}}-f(\bar y)-\frac{\sfd^p(y,\bar y)}{pt^{p-1}}\\
&\leq
\frac{(\sfd(z,y)+\sfd(y,\bar y))^p}{pt^{p-1}}-\frac{\sfd^p(x_i,y_i)}{pt^{p-1}}
\\
&\leq \frac{\sfd(z,y)}{t^{p-1}}\bigl(\sfd(z,y)+D^+(y,t)\bigr)^{p-1},
\end{split}
\]
so that dividing by $\sfd(z,y)$ and inverting the roles of $y$ and $z$ gives
$$
{\rm Lip}\bigl(Q_tf,B_r(x)\bigr)\leq t^{1-p}\bigl(\sup_{y\in B_r(x)}D^+(y,t)\bigr)^{p-1}.
$$
Letting $r\downarrow 0$ and using the upper
semicontinuity of $D^+$ we get \eqref{eq:hjbss}.

Finally, the bound on the Lipschitz constant of $Q_tf$ follows
directly from \eqref{eq:boundD} and \eqref{eq:hjbss}.
\end{proof}

\begin{theorem}[Subsolution of HJ]\label{thm:subsol}
For every $x\in X$ it holds
\begin{equation}\label{eq:hjbsus}
\frac{\d}{\d t}Q_tf(x)+\frac{1}{q}|\nabla Q_tf|^q(x)\leq 0
\end{equation}
for every $t\in (0,\infty)$, with at most countably many exceptions.
\end{theorem}
\begin{proof}
The claim is a direct consequence of
Propositions~\ref{prop:timederivative} and
Proposition~\ref{prop:slopesqt}.
\end{proof}

Notice also that \eqref{eq:hjbss} allows to write the HJ sub solution property
in a stronger form using the asymptotic Lipschitz constant ${\rm Lip}_a(Q_tf,\cdot)$ in place of
$|\nabla Q_t f|$, namely for all $x\in X$ it holds
\begin{equation}\label{eq:hjbsusbis}
\frac{\d}{\d t}Q_tf(x)+\frac{1}{q}({\rm Lip}_a(Q_tf,x))^q\leq 0
\end{equation}
for every $t\in (0,\infty)$, with at most countably many exceptions.

We just proved that in an arbitrary metric space the Hopf-Lax
formula produces subsolutions of the Hamilton-Jacobi equations. In
geodesic spaces this result can be improved to get solutions. Since
we shall not need the result, we just state it (the proof is
analogous to \cite[Proposition~3.6]{Ambrosio-Gigli-Savare11}).

\begin{theorem}[Supersolution of HJ]\label{thm:supersol}
Assume that $(X,\sfd)$ is a geodesic space. Then equality holds in
\eqref{eq:hjbss}. In particular, for all $x\in X$ it holds
\[
\frac{\d}{\d t}Q_tf(x)+\frac{1}{q}|\nabla Q_tf|^q(x)=0
\]
for every $t\in(0,\infty)$, with at most countably many exceptions.
\end{theorem}

\section{Weak gradients}\label{sec:weakgra}

Let $(X,\sfd)$ be a complete and separable metric space and let
$\mm$ be a nonnegative $\sigma$-finite Borel measure in $X$. In this
section we introduce and compare four notions of weak gradients, the
gradient $|\nabla f|_{C,q}$ introduced in \cite{Cheeger00}, the
gradient $|\nabla f|_{S,q}$ introduced in \cite{Koskela-MacManus}
and further studied in \cite{Shanmugalingam00} and the gradients
$\relgradq fq$ and $\weakgradq fq$ whose definition can be obtained
adapting to general power functions the approach of
\cite{Ambrosio-Gigli-Savare11}. We will also see that
\begin{equation}\label{allinequalities}
|\nabla f|_{w,q}\leq|\nabla f|_{S,q}\leq|\nabla f|_{C,q}\leq |\nabla
f|_{*,q}\qquad\text{$\mm$-a.e. in $X$.}
\end{equation}
We shall prove in Section~\ref{sequivalence} that actually all
inequalities are equalities, by proving equality of the two extreme
sides. As in the previous section, we shall denote by $p$ the dual
exponent of $q$.

\subsection{Upper gradients}

Following \cite{Heinonen-Koskela98}, we say that a Borel function
$g$ is an upper gradient of a Borel function $f:X\to\R$ if the
inequality
\begin{equation}\label{eq:uppergradient}
\biggl|\int_{\partial\gamma}f\biggr|\leq\int_\gamma g
\end{equation}
holds for all absolutely continuous curves $\gamma:[0,1]\to X$. Here
$\int_{\partial\gamma} f=f(\gamma_1)-f(\gamma_0)$, while
$\int_\gamma g=\int_0^1g(\gamma_s)|\dot\gamma_s|\,\d s$.

It is well-known and easy to check that the slope is an upper
gradient, for locally Lipschitz functions.

\subsection{Cheeger's gradient $|\nabla f|_{C,q}$}

The following definition is taken from \cite{Cheeger00}, where weak
gradients are defined from upper gradients via a relaxation
procedure.

\begin{definition}[$q$-relaxed upper gradient]\label{def:cheeger00} We say that
$g\in L^q(X,\mm)$ is a $q$-relaxed upper gradient of $f\in
L^q(X,\mm)$ if there exist $\tilde{g}\in L^q(X,\mm)$, functions
$f_n\in L^q(X,\mm)$ and upper gradient $g_n$ of $f_n$ such that:
\begin{itemize}
\item[(a)] $f_n\to f$ in $L^q(X,\mm)$ and $g_n$ weakly converge to
$\tilde{g}$ in $L^q(X,\mm)$;
\item[(b)] $\tilde{g}\leq g$ $\mm$-a.e. in $X$.
\end{itemize}
We say that $g$ is a minimal $q$-relaxed upper gradient of $f$ if
its $L^q(X,\mm)$ norm is minimal among $q$-relaxed upper gradients.
We shall denote by $|\nabla f|_{C,q}$ the minimal $q$-relaxed upper
gradient.
\end{definition}

\subsection{Minimal $q$-relaxed slope $|\nabla f|_{*,q}$}

The second definition of weak gradient we shall consider is a
variant of the previous one and arises by relaxing the integral of
the $q$-th power of the slope of Lipschitz functions. In comparison
with Definition~\ref{def:cheeger00}, we are considering only
Lipschitz approximating functions and we are taking their slopes as
upper gradients. In the spirit of the Sobolev space theory, it
should be considered as an ``$H$ definition'', since an
approximation with Lipschitz functions is involved.

\begin{definition}[Relaxed slope]\label{def:genuppergrad} We say that $g\in L^q(X,\mm)$ is a
$q$-relaxed slope of $f\in L^q(X,\mm)$ if there exist $\tilde{g}\in
L^q(X,\mm)$ and Lipschitz functions $f_n\in L^q(X,\mm)$ such that:
\begin{itemize}
\item[(a)] $f_n\to f$ in $L^q(X,\mm)$ and $|\nabla f_n|$ weakly converge to
$\tilde{g}$ in $L^q(X,\mm)$;
\item[(b)] $\tilde{g}\leq g$ $\mm$-a.e. in $X$.
\end{itemize}
We say that $g$ is the minimal $q$-relaxed slope of $f$ if its
$L^q(X,\mm)$ norm is minimal among $q$-relaxed slopes. We shall
denote by $\relgradq fq$ the minimal $q$-relaxed slope.
\end{definition}

By this definition and the sequential compactness of weak
topologies, any $L^q$ limit of Lipschitz functions $f_n$ with
$\int|\nabla f_n|^q\,\d\mm$ uniformly bounded has a $q$-relaxed
slope. On the other hand, using Mazur's lemma (see
\cite[Lemma~4.3]{Ambrosio-Gigli-Savare11} for details), the
definition of $q$-relaxed slope would be unchanged if the weak
convergence of $|\nabla f_n|$ in (a) were replaced by the condition
$|\nabla f_n|\leq g_n$ and $g_n\to\tilde{g}$ strongly in
$L^q(X,\mm)$. This alternative characterization of $q$-relaxed
slopes is suitable for diagonal arguments and proves, together with
\eqref{eq:subadd}, that the collection of $q$-relaxed slopes is a
closed convex set, possibly empty. Hence, thanks to the uniform
convexity of $L^q(X,\mm)$, the definition of $\relgradq fq$ is well
posed. Also, arguing as in \cite{Ambrosio-Gigli-Savare11} and using once more the uniform
convexity of $L^q(X,\mm)$, it is not difficult to show the following result:

\begin{proposition}\label{prop:easy}
If $f\in L^q(X,\mm)$ has a $q$-relaxed slope then there exist
Lipschitz functions $f_n$ satisfying
\begin{equation}\label{densitylip1}
\lim_{n\to\infty}\int_X|f_n-f|^q\,\d\mm+\int_X\bigl||\nabla
f_n|-|\nabla f|_{*,q}\bigr|^q\,\d\mm=0.
\end{equation}
\end{proposition}

Since the slope is an upper gradient for Lipschitz functions it
turns out that any $q$-relaxed slope is a $q$-relaxed upper
gradient, hence
\begin{equation}\label{allinequalities1}
|\nabla f|_{C,q}\leq |\nabla f|_{*,q}\qquad\text{$\mm$-a.e. in $X$}
\end{equation}
whenever $f$ has a $q$-relaxed slope.

\begin{remark}\label{rem:whyq}
{\rm Notice that in principle the integrability of $f$ could be
decoupled from the integrability of the gradient, because no global
Poincar\'e inequality can be expected at this level of generality.
Indeed, to increase the symmetry with the next two gradients, one
might even consider the convergence $\mm$-a.e. of the approximating
functions, removing any integrability assumption. We have left the
convergence in $L^q$ because this presentation is more consistent
with the usual presentations of Sobolev spaces, and the definitions
given in \cite{Cheeger00} and \cite{Ambrosio-Gigli-Savare11}. Using
locality and a truncation argument, the definitions can be extended
to more general classes of functions, see
\eqref{eq:extendedrelaxed}.}\fr
\end{remark}

\subsection{$q$-upper gradients and $|\nabla f|_{S,q}$}

Here we recall a weak definition of upper gradient, taken from
\cite{Koskela-MacManus} and further studied in
\cite{Shanmugalingam00} in connection with the theory of Sobolev
spaces, where we allow for exceptions in \eqref{eq:uppergradient}.
Recall that, for $\Gamma\subset AC([0,1],X)$, the $q$-modulus ${\rm
Mod}_q(\Gamma)$ is defined by (see \cite{Fuglede} for a systematic
analysis of this concept)
\begin{equation}
\label{eq:defmod2} {\rm
Mod}_q(\Gamma):=\inf\Big\{\int_X\rho^q\,\d\mm: \ \int_\gamma\rho\geq
1\ \ \forall \gamma\in\Gamma\Big\}.
\end{equation}
We say that $\Gamma$ is ${\rm Mod}_q$-negligible if ${\rm
Mod}_q(\Gamma)=0$. Accordingly, we say that a Borel function
$g:X\to[0,\infty]$ is a $q$-upper gradient of $f$ if there exist a
function $\tilde f$ and a ${\rm Mod}_q$-negligible set $\Gamma$ such
that $\tilde{f}=f$ $\mm$-a.e. in $X$ and
\[
\big|\tilde f(\gamma_0)-\tilde f(\gamma_1)\big|\leq\int_\gamma
g\qquad\forall \gamma\in AC([0,1],X)\setminus\Gamma.
\]
It is not hard to prove that the collection of all  $q$-upper
gradients of $f$ is convex and closed, so that we can call minimal
$q$-upper gradient, and denote by $|\nabla f|_{S,q}$, the element
with minimal $L^q(X,\mm)$ norm. Furthermore, the inequality
\begin{equation}\label{allinequalities2}
|\nabla f|_{S,q}\leq|\nabla f|_{C,q}\qquad\text{$\mm$-a.e. in $X$}
\end{equation}
(namely, the fact that all $q$-relaxed upper gradients are $q$-upper
gradients) follows by a stability property of $q$-upper gradients
very similar to the one stated in Theorem~\ref{thm:stabweak} below
for $q$-weak upper gradients, see
\cite[Lemma~4.11]{Shanmugalingam00}. Finally, an observation due to
Fuglede (see Remark~\ref{rem:Fuglede} below) shows that any
$q$-upper gradient can be strongly approximated in $L^q(X,\mm)$ by
upper gradients. This has been used in \cite{Shanmugalingam00} to
show that the equality $|\nabla f|_{S,q}=|\nabla f|_{C,q}$
$\mm$-a.e. in $X$ holds.

\begin{remark}[Fuglede]\label{rem:Fuglede}{\rm
If ${\rm Mod}_q(\Gamma)=0$ and $\eps>0$, then we can find $\rho\in
L^q(X,\mm)$ with $\|\rho\|_q<\eps$ and $\int_\gamma\rho=\infty$ for
all $\gamma\in\Gamma$. Indeed, if we choose functions $\rho_n\in
L^q(X,\mm)$ with $\|\rho_n\|_q<1/n$ and $\int_\gamma\rho_n\geq 1$
for all $\gamma\in\Gamma$, the function
$$
\rho:=\sum_{n\geq 1} \frac\delta{n}\rho_n
$$
has the required property for $\delta=\delta(\eps)>0$ small
enough.}\fr
\end{remark}

\subsection{$q$-weak upper gradients and $|\nabla f|_{w,q}$}

Recall that the evaluation maps $\rme_t:C([0,1],X)\to X$ are defined
by $\rme_t(\gamma):=\gamma_t$. We also introduce the restriction
maps ${\rm restr}_t^s: C([0,1],X)\to C([0,1],X)$, $0\le t\le s\le
1$, given by
\begin{equation}
{\rm restr}_t^s(\gamma)_r:=\gamma_{(1-r)t+rs},\label{eq:93}
\end{equation}
so that ${\rm restr}_t^s$ ``stretches'' the restriction of the curve
to $[s,t]$ to the whole of $[0,1]$.

Our definition of $q$-weak upper gradient still allows for
exceptions in \eqref{eq:uppergradient}, but with a different notion
of exceptional set, see also Remark~\ref{rem:comparenullsets} below.

\begin{definition}[Test plans and negligible sets of curves]\label{def:testplans}
We say that a probability measure $\ppi\in\prob{C([0,1],X)}$ is a
$p$-\emph{test plan} if $\ppi$ is concentrated on $AC^p([0,1],X)$,
$\iint_0^1|\dot\gamma_t|^p\d t\,\d\ppi<\infty$ and there exists a
constant $C(\ppi)$ such that
\begin{equation}
(\e_t)_\#\ppi \leq C(\ppi)\mm\qquad\forall t\in[0,1].
\label{eq:1}
\end{equation}
A Borel set $A\subset C([0,1],X)$ is said to be
$q$-\emph{negligible} if $\ppi(A)=0$ for any $p$-test plan $\ppi$. A
property which holds for every $\gamma\in C([0,1],X)$, except
possibly a $q$-negligible set, is said to hold for $q$-almost every
curve.
\end{definition}
Observe that, by definition, $C([0,1],X)\setminus AC^p([0,1],X)$ is
$q$-negligible, so the notion starts to be meaningful when we look
at subsets $A$ of $AC^p([0,1],X)$. 
\begin{remark}
  \label{re:easy}
  \upshape
  An easy consequence of condition \eqref{eq:1} is that if two
  $\mm$-measurable functions $f,\,g:X\to\R$ coincide up to a
  $\mm$-negligible set and $\mathcal T$ is an at most countable subset
  of $[0,1]$, then the functions
  $f\circ \gamma$ and $g\circ \gamma$ coincide in $\mathcal T$ 
  for $q$-almost every curve
  $\gamma$. 

  Moreover, choosing an arbitrary $p$-test plan $\ppi$ and applying Fubini's
  Theorem to the product measure $\Leb 1\times \ppi$
  in $(0,1)\times C([0,1];X)$ we also obtain that
  $f\circ\gamma=g\circ\gamma$ $\Leb 1$-a.e.\ in $(0,1)$ for
  $\ppi$-a.e.\ curve $\gamma$; since $\ppi$ is arbitrary, the same
  property holds for $q$-a.e.\ $\gamma$.
\end{remark}

Coupled with the definition of
$q$-negligible set of curves, there are the definitions of $q$-weak upper gradient and
 of functions which are Sobolev along $q$-a.e. curve.

\begin{definition}[$q$-weak upper gradients]
A Borel function
$g:X\to[0,\infty]$ is a $q$-weak upper gradient of $f:X\to \R$ if
\begin{equation}
\label{eq:inweak} \left|\int_{\partial\gamma}f\right|\leq
\int_\gamma g\qquad\text{for $q$-a.e. $\gamma$.}
\end{equation}
\end{definition}

\begin{definition}[Sobolev functions along $q$-a.e. curve]
 A function $f:X\to\R$ is Sobolev along $q$-a.e. curve if for
$q$-a.e. curve $\gamma$ the function $f\circ\gamma$ coincides a.e.
in $[0,1]$ and in $\{0,1\}$ with an absolutely continuous map
$f_\gamma:[0,1]\to\R$.
\end{definition}

By Remark \ref{re:easy} applied to $\mathcal T:=\{0,1\}$, \eqref{eq:inweak} does not depend on
the particular representative of $f$ in the class of $\mm$-measurable
function coinciding with $f$ up to a $\mm$-negligible set. 
The same Remark also shows that  the property of being Sobolev along
$q$-q.e.\ curve $\gamma$ is independent of the representative in the
class of $\mm$-measurable functions coinciding with $f$ $\mm$-a.e.\
in $X$.

In the next remark, using Lemma~\ref{lem:Fibonacci}, we prove that the existence of a $q$-weak
upper gradient $g$ such that $\int_\gamma g<\infty$ for
$q$-a.e.\ $\gamma$ (in particular if $g\in L^q(X,\mm)$) implies Sobolev regularity along $q$-a.e.\ curve. Notice that 
only recently we realized that the validity of this implication, compare with the definitions
given in \cite{Ambrosio-Gigli-Savare11}, only apparently stronger.

\begin{remark}[Equivalence with the axiomatization in \cite{Ambrosio-Gigli-Savare11}]
  \label{re:restr}{\rm Notice that if $\ppi$ is a $p$-test plan, so is $({\rm
restr}_t^s)_\sharp\ppi$. Hence if $g$ is a $q$-weak upper gradient
of $f$ such that $\int_\gamma g<\infty$ for
$q$-a.e.\ $\gamma$, then for every $t<s$ in $[0,1]$ it holds
    \[
    |f(\gamma_s)-f(\gamma_t)|\leq \int_t^s g(\gamma_r)|\dot\gamma_r|\,\d
    r \qquad\text{for $q$-a.e. $\gamma$.}
    \]
    Let $\ppi$ be a $p$-test plan: by Fubini's theorem applied
    to the product measure $\Leb2\times\ppi$ in $(0,1)^2\times
    C([0,1];X)$, it follows that for $\ppi$-a.e. $\gamma$ the function
     $f$ satisfies
    \[
    |f(\gamma_s)-f(\gamma_t)|\leq \Bigl|\int_t^s g(\gamma_r)|\dot\gamma_r|\,\d
    r \Bigr|\qquad\text{for $\Leb{2}$-a.e. $(t,s)\in (0,1)^2$.}
    \]
    An analogous argument shows that 
    \begin{equation}
      \label{eq:2}
      \left\{
    \begin{aligned}
      \textstyle |f(\gamma_s)-f(\gamma_0)|&\textstyle 
      \leq \int_0^s
      g(\gamma_r)|\dot\gamma_r|\,\d r\\
      \textstyle |f(\gamma_1)-f(\gamma_s)|&\textstyle \leq \int_s^1
      g(\gamma_r)|\dot\gamma_r|\,\d r
    \end{aligned}\right.
    \qquad\text{for $\Leb{1}$-a.e. $s\in (0,1)$.}
\end{equation}
 Since $g\circ \gamma|\dot \gamma|\in L^1(0,1)$ for
    $\ppi$-a.e.\ $\gamma$,  
    by Lemma~\ref{lem:Fibonacci} it follows that $f\circ\gamma\in W^{1,1}(0,1)$
    for $\ppi$-a.e. $\gamma$, and
    \begin{equation}\label{eq:pointwisewug}
      \biggl|\frac{\d}{\dt}(f\circ\gamma)\biggr|\leq
      g\circ\gamma|\dot\gamma|\quad\text{a.e. in $(0,1)$, for
        $\ppi$-a.e. $\gamma$.}
    \end{equation}
  Since $\ppi$ is arbitrary, we conclude that $f\circ\gamma\in
  W^{1,1}(0,1)$ for $q$-a.e.\ $\gamma$, and therefore it admits an
  absolutely continuous representative $f_\gamma$; moreover,
  by \eqref{eq:2}, it is immediate to check that $f(\gamma(t))=f_\gamma(t)$ for $t\in \{0,1\}$ and $q$-a.e.\ $\gamma$.
    \fr   }
\end{remark}

Using the same argument given in the previous remark it is
immediate to show that if $f$ is Sobolev along $q$-a.e. curve it holds
\begin{equation}\label{eq:locweak}
\text{$g_i$, $i=1,2$ $q$-weak upper gradients of $f$}\quad\Longrightarrow\quad \text{$\min\{g_1,g_2\}$ $q$-weak upper gradient of $f$.}
\end{equation}
Using this stability property we can recover, again, a distinguished
minimal object.

\begin{definition}[Minimal $q$-weak upper gradient]
  Let $f:X\to\R$ be Sobolev along $q$-a.e. curve.
  The minimal $q$-weak upper gradient $\weakgradq fq$ of $f$
  is the $q$-weak upper gradient characterized, up to
$\mm$-negligible sets, by the property
\begin{equation}\label{eq:defweakgrad}
  \weakgradq fq\leq g\qquad\text{$\mm$-a.e. in $X$, for every $q$-weak upper
    gradient $g$ of $f$.}
\end{equation}
\end{definition}

Uniqueness of the minimal weak upper gradient is obvious. For
existence, since $\mm$ is $\sigma$-finite we can find a Borel and
$\mm$-integrable function $\theta:X\to (0,\infty)$ and $\weakgradq
fq :=\inf_n g_n$, where $g_n$ are $q$-weak upper gradients which
provide a minimizing sequence in
$$
\inf\left\{\int_X \theta\, {\rm tan}^{-1}g\,\d\mm:\ \text{$g$ is a
$q$-weak upper gradient of $f$}\right\}.
$$
We immediately see, thanks to \eqref{eq:locweak}, that we can assume
with no loss of generality that $g_{n+1}\leq g_n$. Hence, by
monotone convergence, the function $\weakgradq fq$ is a $q$-weak
upper gradient of $f$ and $\int_X \theta\,{\rm tan}^{-1}g\,\d\mm$ is
minimal at $g=\weakgrad fq$. This minimality, in conjunction with
\eqref{eq:locweak}, gives \eqref{eq:defweakgrad}.

\begin{remark}\label{rem:comparenullsets}{\rm
Observe that for a Borel set $\Gamma\subset C([0,1],X)$ and a test
plan $\ppi$, integrating on $\Gamma$ w.r.t. $\ppi$ the inequality
$\int_\gamma \rho\geq 1$ and then minimizing over $\rho$, we get
$$
\ppi(\Gamma)\leq (C(\ppi))^{1/q}\bigl({\rm
Mod}_q(\Gamma)\bigr)^{1/q}\biggl(\iint_0^1|\dot\gamma|^p\,\d
s\,\d\ppi(\gamma)\biggr)^{1/p},
$$
which shows that any ${\rm Mod}_q$-negligible set of curves is also
$q$-negligible according to Definition~\ref{def:testplans}. This
immediately gives that any $q$-upper gradient is a $q$-weak upper
gradient, so that
\begin{equation}\label{allinequalities3}
|\nabla f|_{w,q}\leq|\nabla f|_{S,q}\qquad\text{$\mm$-a.e. in $X$.}
\end{equation}
}\fr
\end{remark}

Notice that the combination of \eqref{allinequalities1},
\eqref{allinequalities2} and \eqref{allinequalities3} gives
\eqref{allinequalities}.

\section{Some properties of weak gradients}

In order to close the chain of inequalities in
\eqref{allinequalities} we need some properties of the weak
gradients introduced in the previous section. The following locality
lemma follows by the same arguments in \cite{Cheeger00} or adapting
to the case $q\neq 2$ the proof in
\cite[Lemma~4.4]{Ambrosio-Gigli-Savare11}.

\begin{lemma}[Pointwise minimality of $\relgradq fq$]\label{le:local}
Let $g_1,\,g_2$ be two $q$-relaxed slopes of $f$. Then
$\min\{g_1,g_2\}$ is a $q$-relaxed slope as well. In particular, not
only the $L^q$ norm of $\relgradq fq$ is minimal, but also
$\relgradq fq\leq g$ $\mm$-a.e. in $X$ for any relaxed slope $g$ of
$f$.
\end{lemma}

The previous pointwise minimality property immediately yields
\begin{equation}
\label{eq:facile} \relgradq  fq\leq |\nabla f|\qquad\text{$\mm$-a.e.
in $X$}
\end{equation}
for any Lipschitz function $f:X\to\R$.

Also the proof of locality and chain rule is quite standard, see
\cite{Cheeger00} and \cite[Proposition~4.8]{Ambrosio-Gigli-Savare11}
for the case $q=2$ (the same proof works in the general case).

\begin{proposition}[Locality and chain rule]\label{prop:chain}
If $f\in L^q(X,\mm)$ has a $q$-relaxed slope, the following
properties hold.
\begin{itemize}
\item[(a)] $\relgradq hq=\relgradq fq$ $\mm$-a.e. in $\{h=f\}$
whenever $f$ has a $q$-relaxed slope.
\item[(b)] $\relgradq {\phi(f)}q\leq |\phi'(f)|\relgradq fq$ for any $C^1$ and Lipschitz function
$\phi$ on an interval containing the image of $f$. Equality holds if
$\phi$ is nondecreasing.
\end{itemize}
\end{proposition}

Next we consider the stability of $q$-weak upper gradients (as we
said, similar properties hold for $q$-upper gradients, see
\cite[Lemma~4.11]{Shanmugalingam00} but we shall not need them).

\begin{theorem}[Stability w.r.t. $\mm$-a.e. convergence]\label{thm:stabweak}
Assume that $f_n$ are $\mm$-measurable, Sobolev along $q$-a.e. curve
and that $g_n\in L^q(X,\mm)$ are $q$-weak upper gradients of $f_n$. Assume
furthermore that $f_n(x)\to f(x)\in\R$ for $\mm$-a.e. $x\in X$ and
that $(g_n)$ weakly converges to $g$ in $L^q(X,\mm)$. Then $g$ is a
$q$-weak upper gradient of $f$.
\end{theorem}
\begin{proof}
Fix a $p$-test plan $\ppi$ and $\theta\in L^1(X,\mm)$ strictly
positive (its existence is ensured by the $\sigma$-finiteness
assumption on $\mm$). By Mazur's theorem we can find convex
combinations
$$
h_n:=\sum_{i=N_h+1}^{N_{h+1}}\alpha_ig_i\qquad\text{with
$\alpha_i\geq 0$, $\sum_{i=N_h+1}^{N_{h+1}}\alpha_i=1$,
$N_h\to\infty$}
$$
converging strongly to $g$ in $L^q(X,\mm)$. Denoting by $\tilde f_n$
the corresponding convex combinations of $f_n$, $h_n$ are weak upper
gradients of $\tilde f_n$ and still $\tilde f_n\to f$ $\mm$-a.e. in
$X$.

Since for every nonnegative Borel function $\varphi:X\to [0,\infty]$
it holds (with $C=C(\ppi)$)
\begin{align}
  \notag\int\Big(\int_{\gamma}\varphi\Big)\,\d\ppi&=
  \int\Big(\int_0^1 \varphi(\gamma_t)|\dot
  \gamma_t|\,\d t\Big)\,\d\ppi
  \le
  \int\Big(\int_0^1\varphi^q(\gamma_t)\,\d
  t\Big)^{1/q}
  \Big(\int_0^1 |\dot
  \gamma_t|^p\,\d t\Big)^{1/p}\,\d\ppi
  \\&\notag
  \le
  \Big(\int_0^1 \int\varphi^q\,\d(\rme_t)_\sharp\ppi\,\d
  t\Big)^{1/q}
  \Big(\iint_0^1|\dot\gamma_t|^p\,\d t\,\d\ppi\Big)^{1/p}
\\ &\le  \Big(C\int\varphi^q\,\d\mm\Big)^{1/q}  \Big(\iint_0^1|\dot\gamma_t|^p\,\d t\,\d\ppi\Big)^{1/p},
\label{eq:21}
 \end{align}
we obtain, for $\bar C:=C^{1/q}\Big(\iint_0^1|\dot\gamma_t|^p\,\d
t\,\d\ppi\Big)^{1/p}$
\begin{align*}
\int&\biggl(\int_{\gamma}|h_n-g|+\min\{|\tilde{f}_n-f|,\theta\}\biggr)\,\d\ppi\leq
\bar C\Big(\|h_n-g\|_q+ \|\min\{|\tilde{f}_n-f|,\theta\}\|_q\Big)
\to 0.
\end{align*}
By a diagonal argument we can find a subsequence $n(k)$  such that
$$\int_\gamma|h_{n(k)}-g|+\min\{|\tilde{f}_{n(k)}-f|,\theta\}\to 0$$ as
$k\to\infty$ for $\ppi$-a.e. $\gamma$. Since $\tilde{f}_n$ converge
$\mm$-a.e. to $f$ and the marginals of $\ppi$ are absolutely
continuous w.r.t. $\mm$ we have also that for $\ppi$-a.e. $\gamma$
it holds $\tilde{f}_n(\gamma_0)\to f(\gamma_0)$ and
$\tilde{f}_n(\gamma_1)\to f(\gamma_1)$.

If we fix a curve $\gamma$ satisfying these convergence properties,
since $(\tilde{f}_{n(k)})_\gamma$ are equi-absolutely continuous
(being their derivatives bounded by
$h_{n(k)}\circ\gamma|\dot\gamma|$) and a further subsequence of
$\tilde{f}_{n(k)}$ converges a.e. in $[0,1]$ and in $\{0,1\}$ to
$f(\gamma_s)$, we can pass to the limit to obtain an absolutely
continuous function $f_\gamma$ equal to $f(\gamma_s)$ a.e. in
$[0,1]$ and in $\{0,1\}$ with derivative bounded by
$g(\gamma_s)|\dot\gamma_s|$. Since $\ppi$ is arbitrary we conclude
that $f$ is Sobolev along $q$-a.e. curve and that $h$ is a weak
upper gradient of $f$.
\end{proof}

It is natural to ask whether $r$-upper gradients really depend on $r$ or not.
A natural conjecture is the following: let $r\in (1,\infty)$ and $f:X\to\R$ Borel. Assume that $\mm$ is a
finite measure and that $f$ has a $r$-upper gradient in
$L^r(X,\mm)$. Then, for all $q\in (1,r]$, $f$ has a $q$-upper
gradient and $|\nabla f|_{S,q}=|\nabla f|_{S,r}$ $\mm$-a.e. in $X$.

Notice however that the ``converse'' implication, namely
\begin{equation}\label{eq:koskela}
\qquad\text{$f$ has a $q$-upper gradient in
$L^r(X,\mm)$}\,\,\Rightarrow\,\,\text{$f$ has a $r$-upper gradient
in $L^r(X,\mm)$}
\end{equation}
for $1<q<r<\infty$ does not hold in general. A counterexample has
been shown to us by P.Koskela: consider the set $X$ equal to the
union of the first and third quadrant in $\R^2$, and take as
function $f$ the characteristic function of the first quadrant.
Since the collection of all curves passing from the first to the
third quadrant is ${\rm Mod}_2$-negligible (just take, for $\alpha\in (0,1)$,
the family of curves $\rho_\alpha(x)=\alpha|x|^{\alpha-1}$, and let $\alpha\downarrow 0$)
it follows that $f$ has a $2$-upper gradient equal to $0$. On the other hand, $f$ is discontinuous along the pencil of
curves $\gamma_\theta(t):=(2t-1)(\cos\theta,\sin\theta)$ indexed by
$\theta\in [0,\pi/2]$, and since this family of curves is not ${\rm
Mod}_r$-negligible for $r>2$ it follows that \eqref{eq:koskela} fails for $f$.
In order to show that the family of curves is not ${\rm Mod}_r$-negligible for $r>2$,
suffices to notice that $\int_{\gamma_\theta}g\geq 1$ implies
$$
\frac{1}{2}\leq\biggl(\int_0^1 g^r(\gamma_\theta(t))|2t-1|\,\d t\biggr)^{1/r}
\biggl(\int_0^1|2t-1|^{-r'/r}\,\dt\biggr)^{1/r'}.
$$
Since $r>2$ implies $r'/r<1$, integrating both sides in $[0,\pi/2]$ gives 
a lower bound on the $L^r$ norm of $g$ with a positive constant $c(r)$.

In the presence of doubling and a $(1,q)$-Poincar\'e inequality,
\eqref{eq:koskela} holds, following the Lipschitz approximation
argument in Theorem~4.14 and Theorem~4.24 of \cite{Cheeger00} (we
shall not need this fact in the sequel).

\section{Cheeger's functional and its gradient flow}

In this section we assume that $(X,\sfd)$ is complete and separable
and that $\mm$ is a finite Borel measure. As in the previous
sections, $q\in (1,\infty)$ and $p$ is the dual exponent. In order
to apply the theory of gradient flows of convex functionals in
Hilbert spaces, when $q>2$ we need to extend $\relgradq fq$ also to
functions in $L^2(X,\mm)$ (because Definition~\ref{def:genuppergrad}
was given for $L^q(X,\mm)$ functions). To this aim, we denote
$f_N:=\max\{-N,\min\{f,N\}\}$ and set
\begin{equation}\label{eq:mathcalC}
\mathcal C:=\left\{f:X\to\R:\ \text{$f_N$ has a $q$-relaxed slope
for all $N\in\N$}\right\}.
\end{equation}
Accordingly, for all $f\in\mathcal C$ we set
\begin{equation}\label{eq:extendedrelaxed}
\relgradq fq:=\relgradq {f_N}q\qquad\text{$\mm$-a.e. in $\{|f|<N\}$}
\end{equation}
for all $N\in\N$. We can use the locality property in
Proposition~\ref{prop:chain}(a) to show that this definition is well
posed, up to $\mm$-negligible sets, and consistent with the previous
one. Furthermore, locality and chain rules still apply, so we shall
not use a distinguished notation for the new gradient.

Although we work with a stronger definition of weak gradient,
compared to $|\nabla f|_{C,q}$, we call Cheeger's $q$-functional the
energy on $L^2(X,\mm)$ defined by
\begin{equation}\label{def:Cheeger}
\C_q(f):=\frac{1}{q}\int_X |\nabla f|_{*,q}^q \,\d\mm,
\end{equation}
set to $+\infty$ if $f\in L^2(X,\mm)\setminus\mathcal C$.

\begin{theorem} \label{thm:cheeger} Cheeger's $q$-functional
$\C_q$ is convex and lower semicontinuous in $L^2(X,\mm)$.
\end{theorem}
\begin{proof} The proof of convexity is elementary, so we focus on
lower semicontinuity. Let $(f_n)$ be convergent to $f$ in
$L^2(X,\mm)$ and we can assume, possibly extracting a subsequence
and with no loss of generality, that $\C_q(f_n)$ converges to a
finite limit.

Assume first that all $f_n$ have $q$-relaxed slope, so that that
$\relgradq {f_n}q$ is uniformly bounded in $L^q(X,\mm)$. Let
$f_{n(k)}$ be a subsequence such that $\relgradq {f_{n(k)}}q$ weakly
converges to $g$ in $L^q(X,\mm)$. Then $g$ is a $q$-relaxed slope of
$f$ and
$$
\C_q(f)\leq\frac1q\int_X|g|^q\,\d\mm\leq\liminf_{k\to\infty}\frac1q
\int_X|\nabla f_{n(k)}|^q_{*,q}\,\d\mm
=\liminf_{n\to\infty}\C_q(f_n).
$$
In the general case when $f_n\in{\mathcal C}$ we consider the
functions $f^N_n:=\max\{-N,\min\{f,N\}\}$ to conclude from the
inequality $|\nabla f^N_n|_{*,q}\leq|\nabla f_n|_{*,q}$ that
$f^N:=\max\{-N,\min\{f,N\}\}$ has $q$-relaxed slope for any $N\in\N$
and
$$
\int_X|\nabla f^N|_{*,q}^q\,\d\mm\leq \liminf_{n\to\infty}
\int_X|\nabla f^N_n|_{*,q}^q\,\d\mm\leq\liminf_{n\to\infty}
\int_X|\nabla f_n|_{*,q}^q\,\d\mm.
$$
Passing to the limit as $N\to\infty$, the conclusion follows by
monotone convergence.
\end{proof}

\begin{remark}\label{rem:basiclsc} {\rm More generally, the same argument proves the
$L^2(X,\mm)$-lower semicontinuity of the functional
$$
f\mapsto\int_X \frac{|\nabla f|_{*,q}^q}{|f|^\alpha}\,\d\mm
$$
in $\mathcal C$, for any $\alpha>0$. Indeed, locality and chain rule
allow the reduction to nonnegative functions $f_n$ and we can use
the truncation argument of Theorem~\ref{thm:cheeger} to reduce
ourselves to functions with values in an interval $[c,C]$ with
$0<c\leq C<\infty$. In this class, we can again use the chain rule
to prove the identity
$$
\int_X|\nabla f^\beta|^q_{*,q}\,\d\mm= |\beta|^q\int_X\frac{|\nabla
f|_{*,q}^q}{|f|^\alpha}\,\d\mm
$$
with $\beta:=1-\alpha/q$ to obtain the result when $\alpha\neq q$.
If $\alpha=q$ we use a logarithmic transformation.
 }\fr
\end{remark}

Since the finiteness domain of $\C_q$ is dense in $L^2(X,\mm)$ (it
includes bounded Lipschitz functions), the Hilbertian theory of
gradient flows (see for instance \cite{Brezis73},
\cite{Ambrosio-Gigli-Savare08}) can be applied to Cheeger's
functional \eqref{def:Cheeger} to provide, for all $f_0\in
L^2(X,\mm)$, a locally absolutely continuous map $t\mapsto f_t$ from
$(0,\infty)$ to $L^2(X,\mm)$, with $f_t\to f_0$ as $t\downarrow 0$,
whose derivative satisfies
\begin{equation}\label{eq:ODE}
\frac{d}{dt}f_t\in -\partial^-\C_q(f_t)\qquad\text{for a.e. $t\in
(0,\infty)$.}
\end{equation}

Having in mind the regularizing effect of gradient flows, namely the
selection of elements with minimal $L^2(X,\mm)$ norm in
$\partial^-\C_q$, the following definition is natural.

\begin{definition}[$q$-Laplacian]\label{def:delta}
The $q$-Laplacian $\Delta_q f$ of $f\in L^2(X,\mm)$ is defined for
those $f$ such that $\partial^-\C_q(f)\neq\emptyset$. For those $f$,
$-\Delta_q f$ is the element of minimal $L^2(X,\mm)$ norm in
$\partial^-\C_q(f)$. The domain of $\Delta_q$ will be denoted by
$D(\Delta_q)$.
\end{definition}

\begin{remark}[Potential lack of linearity]\label{re:laplnonlin}{\rm
It should be observed that, even in the case $q=2$, in general the
Laplacian is \emph{not} a linear operator. Still, the trivial
implication
\[
v\in\partial^- \C_q(f)\qquad\Longrightarrow\qquad \lambda^{q-1}
v\in\partial^- \C_q(\lambda f),\quad\forall \lambda\in\R,
\]
ensures that the $q$-Laplacian (and so the gradient flow of $\C_q$)
is $(q-1)$-homogenous. }\fr
\end{remark}

We can now write
$$
\frac{\d}{\d t}f_t=\Delta_q f_t
$$
for gradient flows $f_t$ of $\C_q$, the derivative being understood
in $L^2(X,\mm)$, in accordance with the classical case.

\begin{proposition}[Integration by parts]
\label{prop:deltaineq} For all $f\in D(\Delta_q)$, $g\in D(\C_q)$ it
holds
\begin{equation}
\label{eq:delta1} -\int_X g\Delta_q f\,\d\mm\leq \int_X \relgradq
gq|\nabla f|_{*,q}^{q-1}\,\d\mm.
\end{equation}
Equality holds if $g=\phi(f)$ with $\phi\in C^1(\R)$ with bounded
derivative on the image of $f$.
\end{proposition}
\begin{proof}
Since $-\Delta_q f\in\partial^-\C_q(f)$ it holds
\[
\C_q(f)-\int_X \eps g\Delta_q f\,\d\mm\leq \C_q(f+\eps
g),\qquad\forall g\in L^q(X,\mm),\,\,\eps\in\R.
\]
For $\eps>0$, $\relgradq fq+\eps \relgradq gq$ is a $q$-relaxed
slope of $f+\eps g$ (possibly not minimal) whenever $f$ and $g$ have
$q$-relaxed slope. By truncation, it is immediate to obtain from
this fact that $f,\,g\in\mathcal C$ implies $f+\eps g\in\mathcal C$
and
$$
\relgradq {(f+\eps g)}q \leq\relgradq fq+\eps \relgradq
gq\qquad\text{$\mm$-a.e. in $X$.}
$$
Thus it holds $q\C_q(f+\eps g)\leq\int_X(\relgradq fq+\eps\relgradq
gq)^q\,\d\mm$ and therefore
\[
-\int_X\eps g\Delta_q f\,\d\mm\leq \frac1q\int_X(\relgradq
fq+\eps\relgradq gq)^q-|\nabla f|_{*,q}^q\,\d\mm=\eps\int_X\relgradq
gq|\nabla f|^{q-1}_{*,q}\,\d\mm+o(\eps).
\]
Dividing by $\eps$ and letting $\eps\downarrow 0$ we get
\eqref{eq:delta1}.

For the second statement we recall that $\relgradq {(f+\eps
\phi(f))}q=(1+\eps \phi'(f))\relgradq fq$ for $|\eps|$ small enough.
Hence
\[
\C_q(f+\eps \phi(f))-\C_q(f)= \frac{1}{q}\int_X|\nabla
f|_{*,q}^q\bigl((1+\eps \phi'(f))^q-1\bigr)\,\d\mm=\eps\int_X|\nabla
f|_{*,q}^q \phi'(f)\,\d\mm+o(\eps),
\]
which implies that for any $v\in \partial^-\C_q(f)$ it holds
$\int_Xv \phi(f)\,\d\mm=\int_X|\nabla f|_{*,q}^q\phi'(f)\,\d\mm$,
and gives the thesis with $v=-\Delta_q f$.
\end{proof}

\begin{proposition}[Some properties of the gradient flow of $\C_q$]\label{prop:basecal}
Let $f_0\in L^2(X,\mm)$ and let $(f_t)$ be the gradient flow of
$\C_q$ starting from $f_0$. Then the following properties hold.\\*
\noindent (Mass preservation) $\int f_t\,\d\mm=\int f_0\,\d\mm$ for
any $t\geq 0$.\\* \noindent (Maximum principle) If $f_0\leq C$
(resp. $f_0\geq c$) $\mm$-a.e. in $X$, then $f_t\leq C$ (resp
$f_t\geq c$) $\mm$-a.e. in $X$ for any $t\geq 0$.\\* (Energy
dissipation) Suppose $0<c\leq f_0\leq C<\infty$ $\mm$-a.e. in $X$
and $\Phi\in C^2([c,C])$. Then $t\mapsto\int\Phi(f_t)\,\d\mm$ is
locally absolutely continuous in $(0,\infty)$ and it holds
\[
\frac{\d}{\dt}\int \Phi(f_t)\,\d\mm=-\int\Phi''(f_t)|\nabla
f_t|_{*,q}^q\,\d\mm\qquad\text{for a.e. $t\in (0,\infty)$.}
\]
\end{proposition}
\begin{proof} (Mass preservation) Just notice that from \eqref{eq:delta1} we get
\[
\left|\frac{\d}{\dt}\int f_t\,\d\mm\right|=\left|\int
\mathbf{1}\cdot\Delta_q f_t\,\d\mm\right|\leq\int\relgradq{\mathbf
1}q{|\nabla f_t|_{*,q}^q}\,\d\mm=0\quad\text{for a.e. $t>0$},
\]
where $\mathbf 1$ is the function identically equal to 1, which has
minimal $q$-relaxed slope equal to 0 by \eqref{eq:facile}.\\*
(Maximum principle) Fix $f\in L^2(X,\mm)$, $\tau>0$ and, according
to the so-called implicit Euler scheme, let $f^\tau$ be the unique
minimizer of
\[
g\qquad\mapsto\qquad \C_q(g)+\frac{1}{2\tau}\int_X|g-f|^2\,\d\mm.
\]
Assume that $f\leq C$. We claim that in this case $f^\tau\leq C$ as
well. Indeed, if this is not the case we can consider the competitor
$g:=\min\{f^\tau,C\}$ in the above minimization problem. By locality
we get $\C(g)\leq\C(f^\tau)$ and the $L^2$ distance of $f$ and $g$
is strictly smaller than the one of $f$ and $f^\tau$ as soon as
$\mm(\{f^\tau>C\})>0$, which is a contradiction. Starting from
$f_0$, iterating this procedure, and using the fact that the
implicit Euler scheme converges as $\tau\downarrow 0$ (see
\cite{Brezis73}, \cite{Ambrosio-Gigli-Savare08} for details) to the
gradient flow we get the conclusion.\\* (Energy dissipation) Since
$t\mapsto f_t\in L^2(X,\mm)$ is locally absolutely continuous and,
by the maximum principle, $f_t$ take their values in $[c,C]$
$\mm$-a.e., from the fact that $\Phi$ is Lipschitz in $[c,C]$ we get
the claimed absolute continuity statement. Now notice that we have
$\tfrac{\d}{\d t}\int \Phi(f_t) \,\d\mm=\int \Phi'(f_t)\Delta_q
f_t\,\d\mm$ for a.e. $t>0$. Since $\Phi'$ belongs to $C^1([c,C])$,
from \eqref{eq:delta1} with $g=\Phi'(f_t)$ we get the conclusion.
\end{proof}

\section{Equivalence of gradients}\label{sequivalence}

In this section we prove the equivalence of weak gradients. We
assume that $(X,\sfd)$ is compact (this assumption is used to be
able to apply the results of Section~\ref{sec:hopflax} and in
Lemma~\ref{le:kuwada}, to apply \eqref{eq:dualitabase}) and that
$\mm$ is a finite Borel measure, so that the $L^2$-gradient flow of
$\C_q$ can be used.

We start with the following proposition, which relates energy
dissipation to a (sharp) combination of $q$-weak gradients and
metric dissipation in $W_p$.

\begin{proposition}\label{prop:boundweak}
Let $\mu_t=f_t\mm$ be a curve in $AC^p([0,1],(\prob X,W_p))$. Assume
that for some $0<c<C<\infty$ it holds $c\leq f_t\leq C$ $\mm$-a.e.
in $X$ for any $t\in[0,1]$, and that $f_0$ is Sobolev along $q$-a.e.
curve with $\weakgradq{f_0}q\in L^q(X,\mm)$. Then for all $\Phi\in
C^2([c,C])$ convex it holds
\[
\int \Phi(f_0)\,\d\mm-\int\Phi(f_t)\,\d\mm\leq
\frac1q\iint_0^t\bigl(\Phi''(f_0)|\nabla f_0|_{w,q}\bigr)^qf_s\,\d
s\,\d\mm+\frac1p\int_0^t|\dot\mu_s|^p\,\d s\qquad\forall t>0.
\]
\end{proposition}
\begin{proof} Let $\ppi\in\prob{C([0,1],X)}$ be a plan associated to the curve
$(\mu_t)$ as in Proposition~\ref{prop:lisini}. The assumption
$f_t\leq C$ $\mm$-a.e. and the fact that
$\iint_0^1|\dot\gamma_t|^p\,\d
t\,\d\ppi(\gamma)=\int|\dot\mu_t|^p\,\d t<\infty$ guarantee that
$\ppi$ is a $p$-test plan. Now notice that it holds
$\weakgradq{\Phi'(f_0)}q=\Phi''(f_0)\weakgradq{f_0}q$ (it follows
easily from the characterization \eqref{eq:pointwisewug}), thus we
get
\[
\begin{split}
\int \Phi(f_0)-\int\Phi(f_t)\,\d\mm&\leq
\int \Phi'(f_0)(f_0-f_t)\,\d\mm=\int \Phi'(f_0)\circ \e_0-\Phi'(f_0)\circ \e_t\,\d\ppi\\
&\leq\iint_0^t\Phi''(f_0(\gamma_s))\weakgradq{f_0}q(\gamma_s)|\dot\gamma_s|\,\d s\,\d\ppi(\gamma)\\
&\leq\frac1q\iint_0^t\bigl(\Phi''(f_0(\gamma_s))|\nabla
f_0|_{w,q}(\gamma_s)\bigr)^q\,\d s\,\d\ppi(\gamma)
+\frac1p\iint_0^t|\dot\gamma_s|^p\,\d s\,\d\ppi(\gamma)\\
&=\frac1q\iint_0^t\bigl(\Phi''(f_0)|\nabla f_0|_{w,q}\bigr)^qf_s\,\d
s\,\d\mm+\frac1p\int_0^t|\dot\mu_s|^p\,\d s.
\end{split}
\]
\end{proof}

The key argument to achieve the identification is the following
lemma which gives a sharp bound on the $W_p$-speed of the
$L^2$-gradient flow of $\C_q$. This lemma has been introduced in
\cite{Kuwada10} and then used in
\cite{GigliKuwadaOhta10,Ambrosio-Gigli-Savare11} to study the heat
flow on metric measure spaces.

\begin{lemma}[Kuwada's lemma]\label{le:kuwada}
Let $f_0\in L^q(X,\mm)$ and let $(f_t)$ be the gradient flow of
$\C_q$ starting from $f_0$. Assume that for some $0<c<C<\infty$ it
holds $c\leq f_0\leq C$ $\mm$-a.e. in $X$, and that $\int
f_0\,\d\mm=1$. Then the curve $t\mapsto \mu_t:=f_t\mm\in\prob X$ is
absolutely continuous w.r.t. $W_p$ and it holds
\[
|\dot\mu_t|^p\leq\int\frac{|\nabla f_t|_{*,q}^q}{f_t^{p-1}}\,\d
\mm\qquad\text{for a.e. $t\in (0,\infty)$.}
\]
\end{lemma}
\begin{proof}
We start from the duality formula \eqref{eq:dualitabase} (written
with $\varphi=-\psi$)
\begin{equation}\label{eq:dualityQ}
\frac{W_p^p(\mu,\nu)}p=\sup_{\varphi\in{\rm Lip}(X)}\int_X
Q_1\varphi\, d\nu-\int_X\varphi\,d\mu.
\end{equation}
where $Q_t\varphi$ is defined in \eqref{eq:Nicola1} and
\eqref{eq:Nicola2}, so that $Q_1\varphi=\psi^c$. Fix $\varphi\in{\rm
Lip}(X)$ and recall (Proposition~\ref{prop:timederivative}) that the
map $t\mapsto Q_t\varphi$ is Lipschitz with values in $C(X)$, in
particular also as a $L^2(X,\mm)$-valued map.

Fix also $0\leq t<s$, set $\ell=(s-t)$ and recall that since $(f_t)$
is a gradient flow of $\C_q$ in $L^2(X,\mm)$, the map $[0,\ell]\ni
\tau\mapsto f_{t+\tau}$ is absolutely continuous with values in
$L^2(X,\mm)$. Therefore, since both factors are uniformly bounded,
the map $[0,\ell]\ni\tau\mapsto Q_{\frac\tau\ell}\varphi f_{t+\tau}$
is absolutely continuous with values in $L^2(X,\mm)$. In addition,
the equality
\[
\frac{Q_{\frac{\tau+h}\ell}\varphi
f_{t+\tau+h}-Q_{\frac{\tau}\ell}\varphi
f_{t+\tau}}{h}=f_{t+\tau}\frac{Q_{\frac{\tau+h}\ell}-Q_{\frac\tau\ell}\varphi
}{h}+Q_{\frac{\tau+h}\ell}\varphi\frac{ f_{t+\tau+h}-
f_{t+\tau}}{h},
\]
together with the uniform continuity of $(x,\tau)\mapsto
Q_{\frac\tau\ell}\varphi(x)$ shows that the derivative of
$\tau\mapsto Q_{\frac\tau\ell}\varphi f_{t+\tau}$ can be computed
via the Leibniz rule.

We have:
\begin{equation}
\label{eq:step1}
\begin{split}
\int_X Q_1\varphi\,\d\mu_s-\int_X\varphi \,\d\mu_t&
=\int Q_1\varphi f_{t+\ell}\,\d\mm-\int_X\varphi f_t\,\d\mm
=\int_X\int_0^\ell\frac{\d}{\d\tau}\big(Q_{\frac\tau\ell}\varphi f_{t+\tau}\big)d\tau \,\d\mm\\
&\leq\int_X\int_0^\ell -\frac{|\nabla Q_{\frac\tau\ell}\varphi
|^q}{q\ell}f_{t+\tau}+
Q_{\frac\tau\ell}\varphi \Delta_q f_{t+\tau}\,\d\tau \,\d\mm,\\
\end{split}
\end{equation}
having used Theorem~\ref{thm:subsol}.

Observe that by inequalities \eqref{eq:delta1} and \eqref{eq:facile}
we have
\begin{equation}
\label{eq:sarannouguali}
\begin{split}
\int_X Q_{\frac\tau\ell}\varphi \Delta_q f_{t+\tau} \,\d\mm& \leq
\int_X\relgradq{Q_{\frac\tau\ell}\varphi}q|\nabla
f_{t+\tau}|_{*,q}^{q-1}\,\d\mm\leq \int_X|\nabla
Q_{\frac\tau\ell}\varphi||\nabla f_{t+\tau}|_{*,q}^{q-1} \,\d\mm\\
&\leq \frac1{q\ell}\int_X|\nabla Q_{\frac\tau\ell}\varphi
|^qf_{t+\tau}d\mm+\frac{\ell^{p-1}} p\int_X\frac{|\nabla
f_{t+\tau}|_{*,q}^q}{f_{t+\tau}^{p-1}}\,\d\mm.
\end{split}
\end{equation}
Plugging this inequality in \eqref{eq:step1}, we obtain
\[
\int_X Q_1\varphi \,\d\mu_s-\int_X\varphi \,\d\mu_t\leq
\frac{\ell^{p-1}} p\int_0^\ell\int_X\frac{|\nabla
f_{t+\tau}|_{*,q}^q}{f_{t+\tau}^{p-1}}\,\d\mm.
\]
This latter bound does not depend on $\varphi$, so from
\eqref{eq:dualityQ} we deduce
\[
W_p^p(\mu_t,\mu_s)\leq \ell^{p-1}\int_0^\ell\int_X\frac{|\nabla
f_{t+\tau}|_{*,q}^q}{f^{p-1}_{t+\tau}}\,\d\mm.
\]
At Lebesgue points of $r\mapsto\int_X|\nabla
f_r|_{*,q}^q/f_r^{p-1}\,\d\mm$ where the metric speed exists we
obtain the stated pointwise bound on the metric speed.
\end{proof}

The following result provides equivalence between weak and relaxed
gradients. Recall that the set $\mathcal C$ was defined in
\eqref{eq:mathcalC}.

\begin{theorem}\label{thm:graduguali}
Let $f:X\to\R$ Borel. Assume that $f$ is Sobolev along $q$-a.e.
curve and that $\weakgradq fq\in L^q(X,\mm)$. Then $f\in \mathcal C$
and $\relgradq fq=\weakgradq fq$ $\mm$-a.e. in $X$.
\end{theorem}
\begin{proof}
Up to a truncation argument and addition of a constant, we can
assume that $0<c\leq f\leq C<\infty$ $\mm$-a.e. for some $0<c\leq
C<\infty$. Let $(g_t)$ be the $L^2$-gradient flow of $\C_q$ starting
from $g_0:=f$ and let us choose $\Phi\in C^2([c,C])$ in such a way
that $\Phi''(z)=z^{1-p}$ in $[c,C]$. Recall that $c\leq g_t\leq C$
$\mm$-a.e. in $X$ and that from Proposition~\ref{prop:basecal} we
have
\begin{equation}\label{eq:Amerio}
\int\Phi(g_0)\,\d\mm-\int\Phi(g_t)\,\d\mm=\int_0^t\int_X\Phi''(g_s)|\nabla
g_s|_{*,q}^q\d\mm\,\d s\qquad\forall t\in [0,\infty).
\end{equation}
In particular this gives that $\int_0^\infty\int_X\Phi''(g_s)|\nabla
g_s|_{*,q}^q\,\d\mm\,\d s$ is finite. Setting $\mu_t=g_t\mm$,
Lemma~\ref{le:kuwada} and the lower bound on $g_t$ give that
$\mu_t\in AC^p\bigl((0,\infty),(\prob X,W_p)\bigr)$, so that
Proposition~\ref{prop:boundweak} and Lemma~\ref{le:kuwada} yield
\[
\int \Phi(g_0)\,\d\mm-\int
\Phi(g_t)\,\d\mm\leq\frac1q\int_0^t\int_X\bigl(\Phi''(g_0)|\nabla
g_0|_{w,q}\bigr)^q g_s\,\d\mm\,\d s+\frac1p\int_0^t\int_X\frac{|\nabla
g_s|_{*,q}^q}{g_s^{p-1}}\,\d\mm\,\d s.
\]
Hence, comparing this last expression with \eqref{eq:Amerio}, our
choice of $\Phi$ gives
\[
\frac1q\iint_0^t\,\frac{|\nabla g_s|_{*,q}^q}{g_s^{p-1}}\d
s\,\d\mm\leq\int_0^t\int_X\frac1q \bigl(\frac{|\nabla
g_0|_{w,q}}{g_0^{p-1}}\bigr)^q g_s\,\d\mm\,\d s.
\]
Now, the bound $f\geq c>0$ ensures $\Phi''(g_0)|\nabla g_0|_{*,q}\in
 L^q(X,\mm)$. In addition, the
maximum principle together with the convergence of $g_s$ to $g_0$ in
$L^2(X,\mm)$ as $s\downarrow 0$ grants that the convergence is also
weak$^*$ in $L^\infty(X,\mm)$, therefore
\[
\limsup_{t\downarrow 0}\frac{1}{t}\iint_0^t\,\frac{|\nabla
g_s|_{*,q}^q}{g_s^{p-1}}\d s\,\d\mm\leq\int_X \frac{|\nabla
g_0|_{w,q}^q}{g_0^{q(p-1)}}g_0\d\mm=\int_X \frac{|\nabla
g_0|_{w,q}^q}{g_0^{p-1}}\,\d\mm.
\]
The lower semicontinuity property stated in
Remark~\ref{rem:basiclsc} with $\alpha=p-1$ then gives
\[
\int_X \frac{|\nabla g_0|_{*,q}^q}{g_0^{p-1}}\,\d\mm\leq \int_X
\frac{|\nabla g_0|_{w,q}^q}{g_0^{p-1}}\,\d\mm.
\]
This, together with the inequality $\weakgradq {g_0}q\leq\relgradq
{g_0}q$ $\mm$-a.e. in $X$, gives the conclusion.
\end{proof}

In particular, taking into account \eqref{allinequalities}, we
obtain the following equivalence result. We state it for
$L^q(X,\mm)$ functions because in the definition of $q$-relaxed
upper gradient and $q$-relaxed slope this integrability assumption
is made (see also Remark~\ref{rem:whyq}), while no integrability is
made in the other two definitions. It is also clear that if we
extend the ``relaxed'' definitions of gradient by truncation, as in
\eqref{eq:extendedrelaxed}, then equivalence goes beyond
$L^q(X,\mm)$ functions.

\begin{theorem}[Equivalence of weak gradients] 
 \label{thm:gradugualibis} Let $f\in L^q(X,\mm)$. Then the
following four properties are equivalent:
\begin{itemize}
\item[(i)] $f$ has a $q$-relaxed upper gradient;
\item[(ii)] $f$ has a $q$-relaxed slope;
\item[(iii)] $f$ has a $q$-upper gradient in $L^q(X,\mm)$;
\item[(iv)] $f$ has a $q$-weak upper gradient in $L^q(X,\mm)$.
\end{itemize}
In addition, the minimal $q$-relaxed upper gradient, the minimal
$q$-relaxed slope, the minimal $q$-upper gradient and the minimal
$q$-weak upper gradient coincide $\mm$-a.e. in $X$. 
\end{theorem}
\begin{proof} If either of the four properties holds for some gradient $g$,
then \eqref{allinequalities} gives that $f$ is Sobolev along
$q$-a.e. curve and $|\nabla f|_{w,q}\leq g$ $\mm$-a.e. in $X$. Then,
Theorem~\ref{thm:graduguali} yields $|\nabla f|_{*,q}\leq g$
$\mm$-a.e. in $X$ and we can invoke \eqref{allinequalities} again to
obtain that all four properties hold and the corresponding weak
gradients are equal. 
\end{proof}

\section{Further comments and extensions}\label{sextensions}

In this section we point out how our main results, namely
Theorem~\ref{thm:graduguali} and Theorem~\ref{thm:gradugualibis} can
be extended to more general metric measure spaces. Recall that, in
the previous section, we derived them under the assumptions that
$(X,\sfd)$ is a compact metric space and that $\mm$ is a finite
measure.

\subsection{The role of the compactness assumption in Section~\ref{sec:hopflax}}

The compactness assumption is not really needed, and suffices to
assume that $(X,\sfd)$ is a complete metric space. The only
difference appears at the level of the definition of $D^\pm(x,t)$,
since in this case existence of minimizers is not ensured and one
has to work with minimizing sequences. This results in longer
proofs, but the arguments remain essentially the same, see
\cite{Ambrosio-Gigli-Savare11} for a detailed proof in the case
$p=q=2$. Thanks to this remark, the proof of the equivalence results
immediately extends to complete and separable metric measure spaces
with $(X,\sfd,\mm)$ with $\sfd$ bounded and $\mm$ finite.

Also, it is worthwhile to remark that all results (except of course
the Lipschitz bounds on $Q_tf$ and the continuity of $t\mapsto Q_tf$
from $[0,\infty)$ to $C(X)$) of Section~\ref{sec:hopflax} remain
valid for lower semicontinuous functions $f:X\to\R\cup\{+\infty\}$
satisfying
$$
f(x)\geq -C\bigl(1+\sfd^r(x,\bar x)\bigr)\qquad\forall x\in X
$$
for suitable $\bar x\in X$, $C\geq 0$, $r\in [0,p)$.

\subsection{Locally finite metric measure spaces}

We say that a metric measure space $(X,\sfd,\mm)$ is locally finite
if $(X,\sfd)$ is complete and separable and any $x\in{\rm
supp\,}\mm$ has a neighbourhood $U$ with finite $\mm$-measure.

For any locally finite metric measure space it is not difficult to
find (choosing for instance as $U$ balls with $\mm$-negligible
boundary) a nondecreasing sequence of open sets $A_h$ whose union
covers $\mm$-almost all of $X$ and whose boundaries $\partial A_h$
are $\mm$-negligible. Then, setting $X_h=\overline{A_h}$, we can
apply the equivalence results in all metric measure spaces
$(X_h,\sfd,\mm)$ to obtain the equivalence in $(X,\sfd,\mm)$. This
is due to the fact that the minimal $q$-weak upper gradient
satisfies this local-to-global property (see
\cite[Theorem~4.20]{Ambrosio-Gigli-Savare11bis} for a proof in the
case $p=q=2$):
\begin{equation}\label{eq:locglob1}
|\nabla f|_{X,w,q}=|\nabla f|_{X_h,w,q}\qquad\text{$\mm$-a.e. in
$X_h$.}
\end{equation}
An analogous property holds for the larger gradient, namely the minimal
$q$-relaxed slope (arguing as in \cite[Lemma 4.11]{Ambrosio-Gigli-Savare11}):
\begin{equation}\label{eq:locglob2}
|\nabla f|_{X,*,q}=|\nabla f|_{X_h,*,q}\qquad\text{$\mm$-a.e. in
$X_h$.}
\end{equation}
Combining \eqref{eq:locglob1} and \eqref{eq:locglob2} gives
the identification result for all gradients and all locally finite metric measure spaces.

\subsection{An enforcement of the density result}\label{sec:improveslope}

In Theorem~\ref{thm:graduguali} we proved that
if $f:X\to\R$ is Borel and $f$ is Sobolev along $q$-a.e. curve
and $\weakgradq fq\in L^q(X,\mm)$, then there exist Lipschitz functions
$f_n$ convergent to $f$ $\mm$-a.e. in $X$ and satisfying
\begin{equation}\label{eq:lavuolenicola}
|\nabla f_n|\to \weakgradq fq\qquad\text{in $L^q(X,\mm)$.}
\end{equation}
This follows by a diagonal argument, thanks to the fact that all truncations $f_N$ of $f$ satisfy $\C_q(f_N)\leq\tfrac1q
\int_X\weakgradq fq^q\,\d\mm$. It is worthwhile to notice that \eqref{eq:lavuolenicola}  
can be improved asking the existence of Lipschitz functions $f_n$ such that
${\rm Lip}_a(f_n,\cdot)\to\weakgradq fq$ in $L^q(X,\mm)$, where ${\rm Lip}_a(f,\cdot)$ is the asymptotic Lipschitz constant
defined in \eqref{eq:asymlip}: the key observation is
that, as noticed in \eqref{eq:hjbsusbis}, the Hamilton-Jacobi subsolution property holds with the new, and larger, pseudo gradient
${\rm Lip}_a(g,\cdot)$. Starting from this observation, and using the convexity inequality
$$
{\rm Lip}_a\bigl((1-\chi)f+\chi g\bigr)\leq\bigl(1-\chi(x)\bigr){\rm Lip}_a(f,x)+\chi(x){\rm Lip}_a(g,x)+{\rm Lip}(\chi)|f(x)-g(x)|
$$
for $\chi:X\to [0,1]$ Lipschitz and $f,\,g:X\to\R$ Lipschitz,
one can build Cheeger's energy by minimizing the integrals of ${\rm Lip}_a(f_n,\cdot)$ instead of the integral of $|\nabla g|$, 
still getting a convex and lower semicontinuous functional and a corresponding relaxed gradient.
Then, \eqref{eq:hjbsusbis} provides Kuwada's Lemma~\ref{le:kuwada} for the new Cheeger energy and the proof of Theorem~\ref{thm:graduguali} 
can repeated word by word.

\subsection{Orlicz-Wasserstein spaces}

Another potential extension, that we shall not develop here, is for
general Lagrangians-Hamiltonians: one can consider the functions
$$
Q_tf(x):=\inf_{y\in X} f(y)+tL\bigl(\frac{\sfd(y,x)}{t}\bigr)
$$
and prove that $\tfrac{\d}{\d t}Q_tf+H(\nabla Q_tf)\leq 0$ with
$H=L^*$. This way, also gradients in Orlicz spaces as $LLogL$ could
be considered. On the other hand, the Orlicz-Wasserstein distances
$$
W_L(\mu,\nu):=\inf\left\{\lambda>0:\
\inf_{\sppi\in\Gamma(\mu,\nu)}\int
L\bigl(\frac{\sfd(x,y)}{\lambda}\bigr)\,\d\ppi\leq 1\right\}
$$
have not been considered much so far (except in \cite{Sturm-Orlicz}
and more implicitly in \cite{FigalliGangbo,Villani09}) and the
extension of Lisini's superposition theorem to this class of
distances is not known, although expected to be true. These
extensions might be particularly interesting to deal with the
limiting case $q\downarrow 1$, where the Wasserstein exponent $p$
goes to $\infty$ (for instance $LlogL$ integrability of gradients
corresponds to exponential integrability of metric derivative on
curves) .

\subsection{$W^{1,1}$ and $BV$ spaces}

In this subsection we discuss the limiting case $q=1$, $p=\infty$
and assume for the sake of simplicity that $(X,\sfd)$ is locally
compact and separable. Following the approach in \cite{Miranda03}, for
any open set $A\subset X$ we can define
$$
|Df|(A):=\inf\left\{\liminf_{h\to\infty}\int_A|\nabla f_h|\,\d\mm:\
f_h\in {\rm Lip}_{\rm loc}(A),\,\,f_h\to f\,\,\text{in $L^1_{\rm
loc}(A)$}\right\}.
$$
It is possible to show that, whenever $|Df|(X)<\infty$, the set
function $A\mapsto|Df|(A)$ is the restriction to open sets of $X$ of
a finite Borel measure, that we still denote by $|Df|$. In the case
when $|Df|$ is abolutely continuous with respect to $\mm$,
corresponding to the Sobolev space $W^{1,1}$ we may define $|\nabla
f|_{*,1}$ as the density of $|Df|$ w.r.t. $\mm$.

This approach corresponds to $1$-relaxed slopes. Coming to $1$-weak
upper gradients, it is natural to consider $\infty$-test plans as
probability measures $\ppi$ concentrated on Lipschitz curves and to
define exceptional sets of curves using this class of test plans.
Then the class of functions which are $BV$ along $1$-almost every
curve can be defined. It is not hard to show that if
$|Df|(X)<\infty$ and $\ppi$ is a $\infty$-test plan such that
$(\e_t)_\#\ppi\leq C(\ppi)\mm$ for all $t\in [0,1]$ then the following
inequality between measures in $X$ holds:
$$
\int \gamma_\sharp|D(f\circ\gamma)|\,d\ppi(\gamma)\leq C(\ppi)\|{\rm
Lip}(\gamma)\|_{L^\infty(\sppi)}|Df|,
$$
where $|D(f\circ\gamma)|$ is the total variation measure of the map
$f\circ\gamma:[0,1]\to\R$. This provides one connection between
$1$-weak upper gradients and $1$-relaxed slopes, while in \cite{Ambrosio-DiMarino12}
the arguments of this paper are adapted to show that the supremum of
$$
\frac{1}{C(\ppi)\|{\rm
Lip}(\gamma)\|_{L^\infty(\sppi)}}\int \gamma_\sharp|D(f\circ\gamma)|\,d\ppi(\gamma)
$$
in the lattice of measures coincides with $|Df|$.

\def\cprime{$'$} \def\cprime{$'$}

\end{document}